\documentclass[english,letterpaper,11pt,reqno]{amsart}

\usepackage{appendix}
\usepackage{amsbsy}
\usepackage{amsfonts}
\usepackage{amsmath}
\usepackage{amssymb}
\usepackage{amsthm}
\usepackage{graphicx}
\usepackage{ifthen}
\usepackage{textcomp}
\usepackage{enumitem, color, amssymb}
\usepackage{dsfont}
\usepackage{mathrsfs}
\usepackage{calrsfs}
\usepackage[bookmarksnumbered,colorlinks]{hyperref}
\hypersetup{colorlinks=true, linkcolor=blue, citecolor=blue}
\newcommand{\lra}{\longrightarrow}

\newcommand{\RR}{\mathbb{R}}

\newcommand{\vep}{\varepsilon}

\newcommand{\uu}{\mathcal{U}}
\newcommand{\tuu}{\widetilde{\uu}}

\newtheorem{thm}{Theorem}
\newtheorem{lemma}{Lemma}
\newtheorem{cor}{Corollary}

\newcommand{\beqa}{\begin{eqnarray}}
\newcommand{\beq}{\begin{equation}}
\newcommand{\eeqa}{\end{eqnarray}}
\newcommand{\eeq}{\end{equation}}

\newcommand\ip[2]{\langle {#1},{#2}\rangle} 
\newcommand\vv[1]{{\boldsymbol {\it #1}}} 

\newcommand\mb{\overline{\boldsymbol m}}
\newcommand\mm{{\boldsymbol m}}
\newcommand\kk{{\boldsymbol k}}

\newcommand\xx{\vv{x}}
\newcommand\yy{\vv{y}}
\newcommand\cd[2]{\nabla_{\!#1}{#2}}

\newcommand\nb{\overline{\boldsymbol n}}
\newcommand\nn{{\boldsymbol n}}

\newcommand\Ric{\widehat{\text{Ric}}}

\DeclareMathOperator\Div{\mathrm{div}}

\thanks{Mathematics Subject Classification: 53C20, 53C22, 53C25}
\begin{document}
\title[]{On the principal Ricci curvatures of a Riemannian 3-manifold}
\author[]{Amir Babak Aazami and Charles M. Melby-Thompson}
\address{Clark University\hfill\break\indent
Worcester, MA 01610, USA}
\email{aaazami@clarku.edu}

\address{Fudan University\hfill\break\indent
Shanghai 200433, China}
\email{charlesmelby@gmail.com}

\begin{abstract}
We study global obstructions to the eigenvalues of the Ricci tensor on a Riemannian 3-manifold.  As a topological obstruction, we first show that if the 3-manifold is closed, then certain choices of the eigenvalues are prohibited: in particular, there is no Riemannian metric whose corresponding Ricci eigenvalues take the form $(-\mu,f,f)$, where $\mu$ is a positive constant and $f$ is a smooth positive function.  We then concentrate on the case when one of the eigenvalues is zero.  Here we show that if the manifold is complete and its Ricci eigenvalues take the form $(0,\lambda,\lambda)$, where $\lambda$ is a positive constant, then its universal cover must split isometrically.  If the manifold is closed, scalar-flat, and its zero eigenspace contains a unit length vector field that is geodesic and divergence-free, then the manifold must be flat.  Our techniques also apply to the study of Ricci solitons in dimension three.
\end{abstract}
\maketitle

\section{Introduction}
\label{sec:intro}
Let $(M,g)$ be a Riemannian 3-manifold.  Because the Weyl tensor vanishes in dimension three, the curvature of $(M,g)$ is completely determined by its Ricci tensor.  We can associate to this Ricci tensor a smooth endomorphism of the tangent bundle $TM$,
denoted $\Ric\colon TM \lra TM$,
whose fiberwise action $X \mapsto \Ric_p(X)$ is the unique vector satisfying 
$$
g(\,\Ric_p(X),Y)\ =\ \text{Ric}_p(X,Y)\hspace{.1in}\forall Y \in T_pM.
$$
The eigenvalues of $\Ric$ are known as the \emph{principal Ricci curvatures.}  This paper is motivated by the following question: \emph{On a Riemannian 3-manifold, are there global obstructions to the possible values of the principal Ricci curvatures?}  In addition, we attack this problem using a new tool adapted from gravitational physics: namely, the \emph{Newman-Penrose formalism} \cite{newpen62}.  Although this originated in the study of gravitational physics on Lorentzian 4-manifolds, nevertheless, as we show in this paper, it is in fact ideally suited to studying the \emph{divergence, twist,} and \emph{shear} of constant length vector fields on \emph{Riemannian} 3-manifolds, properties which we define in Section \ref{sec:formalism} below.  Therefore a secondary goal of our paper is to promote the further study of this formalism in the study of Riemannian 3-manifolds.     
\vskip 6pt
Returning to the problem of principal Ricci curvatures, the reason we are interested in the global case is because there are in fact no \emph{local} obstructions: it has been shown in \cite{spiro}, and later, by a different method, in \cite{kowalski94}, that given any three smooth functions $f_1,f_2,f_3$ on $\RR^3$, there always exists a Riemannian metric on $\RR^3$ whose corresponding principal Ricci curvatures are $(f_1,f_2,f_3)$.  As an example of a global, topological obstruction, our first result shows that this is no longer the case if the 3-manifold is closed:

\begin{thm}
\label{thm:1main}
On a closed 3-manifold $M$, there is no Riemannian metric with principal Ricci curvatures $-\mu,f,f$ when $\mu$ is a positive constant and $f$ is a positive smooth function on $M$.
\end{thm}

This generalizes a result in \cite{yamato91}, which proved the case when $f$ is constant.  In fact the signature $(-\!+\!+)$ has an interesting history.  In \cite{milnor76}, J. Milnor showed that three-dimensional Lie groups with left invariant Riemannian metrics furnish examples of manifolds with principal Ricci curvatures of globally fixed signature\,---\,\emph{except} for the signatures $(-\!+\!+)$, $(0\!+\!+)$, and $(0\!+\!-)$.  Indeed, we next examine the signature $(0+-)$, in the special case when the manifold is scalar flat, which is to say, when the nonzero principal Ricci curvatures have the same magnitude: $0,f,-f$.  The obstruction in this case concerns the vector field spanning the zero eigenspace:

\begin{thm}
\label{thm:2main}
Let $(M,g)$ be a closed, scalar flat Riemannian 3-manifold.  If the zero eigenspace of the Ricci tensor has a unit length vector field with geodesic and divergence-free flow, then $(M,g)$ is flat.
\end{thm}

To interpret Theorem \ref{thm:2main}, note that if a scalar flat Riemannian 3-manifold has zero as a principal Ricci curvature, then its only possible signatures are $(0\!+\!-)$ or $(000)$.  In fact our techniques enable us to show that if the vector field in question has geodesic flow but is not necessarily divergence-free, then the principal Ricci curvatures may be of the form $(0\!+\!-)$, but they cannot be constants.  Theorem \ref{thm:2main} is the most technical of our results, and in fact it relies in a crucial way on a very deep result from contact geometry: namely, the positive resolution of the \emph{Weinstein Conjecture} in dimension 3, by C. Taubes, in \cite{taubes07}.  Briefly, recall that a \emph{contact form} on a smooth manifold $M$ is a one-form $\theta$ such that $d\theta$ is nondegenerate on the kernel of $\theta$, where $d$ is the exterior derivative.  The \emph{Reeb vector field} of $\theta$ is a the unique vector field $X$ satisfying $\theta(X) = 1$ and $d\theta(X,\cdot) = 0$.  The Weinstein conjecture then states that on a closed manifold, the Reeb vector field of any contact form has a closed orbit.  In Theorem \ref{thm:2main}, the one-form to which we apply the Weinstein Conjecture is of the form $g(X,\cdot)$ for some unit length vector field $X$ with geodesic flow: $\nabla_{\!X} X = 0$.  In such a case, $X$ itself will be the Reeb field of this one-form.  (In fact Theorem \ref{thm:2main} is inspired by \cite{brans75}, as presented in \cite[p. 348]{o1995}, which showed that if a Ricci-flat Lorentzian 4-manifold has a diagonalizable curvature operator with a zero eigenvalue, then it is flat.)
\vskip 6pt
Our techniques allow us to also say something about the signature $(0\!+\!+)$, in the case when the two positive eigenvalues are equal.  In fact this case is distinguished from the others, for the following reason.  Let $(M,g)$ be a Riemannian 3-manifold whose principal Ricci curvatures are $(0,f,f)$, where $f$ is a smooth positive function.  Let $\kk$ be a nowhere vanishing vector field that spans the zero eigenspace.  Then it is straightforward to show that the case $(0,f,f)$ is equivalent to $R(\kk,\cdot,\cdot,\cdot) = 0$, where $R$ is the Riemann curvature 4-tensor; in particular, a 2-plane has zero sectional curvature if and only if it contains $\kk$.  This is an example of a manifold with \emph{conullity 2}.  Specifically, the \emph{nullity space} of $(M,g)$ is defined at each $p \in M$ to be the subspace $\{X \in T_pM : R_p(X,\cdot,\cdot,\cdot) = 0\}$, and the dimension of this subspace is called the nullity of $M$ at $p$.  For such a 3-manifold, our techniques allow us to recover a result that follows from results established in both \cite{szabo85} and \cite{schmidt11}.\footnote{We thank both Wolfgang Ziller and Benjamin Schmidt for bringing, respectively, \cite{szabo85} and \cite{schmidt11} to our attention.}  Namely, if $(M,g)$ is complete and $f$ is a positive constant, so that $(M,g)$ has constant positive scalar curvature, then its universal cover must split isometrically:

\begin{thm}
\label{thm:main} (\cite{szabo85,schmidt11})
Let $(M,g)$ be a complete Riemannian 3-manifold with constant positive scalar curvature.  If there exists a nowhere vanishing vector field $\kk$ satisfying $R(\kk,\cdot,\cdot,\cdot) = 0$, then the universal cover of $(M,g)$ splits isometrically as $\RR \times \widetilde{N}$.
\end{thm}

The unifying feature of these three theorems is the existence of a preferred unit length vector field: the one spanning the negative eigenspace in $(-\!+\!+)$ and the ones spanning the zero eigenspaces in $(0\!+\!\pm)$.  Our method of proof is to examine the properties of this preferred vector field using the \emph{Newman-Penrose formalism} \cite{newpen62} (see also \cite[Chapter 5]{o1995}), which was adapted to Riemannian 3-manifolds in \cite{AA14}.  Indeed, another geometric situation in which a vector field is distinguished is that of a Ricci soliton.  We therefore close our paper with a few results meant to illustrate that the Newman-Penrose formalism is naturally adapted to this setting as well.  Recall that a Riemannian manifold $(M,g)$ is a \emph{Ricci soliton} if there exists a vector field $\kk$ and a constant $\lambda$ such that
\beqa
\label{eqn:rs}
\text{Ric} + \frac{1}{2}\mathscr{L}_{\kk}\,g\ =\ \frac{\lambda}{2} g,
\eeqa
where $\mathscr{L}$ is the Lie derivative.  For a recent survey of Ricci solitons and their relation to the Ricci flow, see \cite{cao10}.  The Ricci soliton is said to be \emph{shrinking, steady,} or \emph{expanding} if, respectively, $\lambda > 0, \lambda = 0$, or $\lambda < 0$.  To begin with, a simple application of our techniques shows that requiring $\kk$ to have unit length forces a three-dimensional Ricci soliton, if nontrivial, to be shrinking:

\begin{lemma}
\label{lemma:Ricci}
Let $(M,g)$ be a Riemannian 3-manifold and $(M,g,\kk)$ a Ricci soliton.  If $\kk$ is a unit length vector field, then either the Ricci soliton is shrinking or else $\kk$ is parallel and $(M,g)$ is flat.
\end{lemma} 

Now relax the condition that $\kk$ have unit length, but place it in the kernel of the endomorphism $\Ric$.  If the Riemannian 3-manifold is closed with constant scalar curvature, then the following is true:

\begin{lemma}
\label{lemma:Ricci2}
Let $(M,g)$ be a closed Riemannian 3-manifold with constant scalar curvature $S$.  If the zero eigenspace of the Ricci tensor contains a nowhere vanishing vector field $\kk$ for which $(M,g,\kk)$ is a Ricci soliton, then $\lambda = S$.  If $S = 0$, then $(M,g)$ is flat.
\end{lemma}

In particular, if $S$ is positive or negative, then the Ricci soliton in Lemma \ref{lemma:Ricci2} must be shrinking or expanding, respectively.  Finally, if we allow for the conullity 2 condition $R(\kk,\cdot,\cdot,\cdot) = 0$, then non-Einstein Ricci solitons are guaranteed to exist by Theorem \ref{thm:main}, but the choice of 3-manifold is restricted:

\begin{cor}
\label{lemma:main3}
Let $(M,g)$ be a simply connected, complete Riemannian 3-manifold with constant positive scalar curvature $S$.  Let $\kk \in \mathfrak{X}(M)$ be nowhere vanishing and satisfy $R(\kk,\cdot,\cdot,\cdot) = 0$.  Then $(M,g)$ splits isometrically as $\RR \times S^2$ and there is a smooth function $f$ such that $(M,g,f\kk)$ is a non-Einstein shrinking Ricci soliton with $\lambda = S$.
\end{cor}

\section{Formalism and Conventions}
\label{sec:formalism}
In this section we give an overview of the Newman-Penrose formalism for Riemannian 3-manifolds, as established in \cite{AA14}.  In what follows, we use the notation ``$\ip{\,}{}$" to denote the metric $g$, and our sign convention for the Riemann tensor is
$$
R(X,Y)Z\ =\ \cd{X}{\cd{Y}Z} - \cd{Y}{\cd{X}Z} - \cd{[X,Y]}{Z}.
$$
Given a local orthonormal frame in $(M,g)$ of the form $\{\kk,\xx,\yy\}$, begin by combining $\xx$ and $\yy$ into complex-valued vector fields
\beqa
\label{eqn:m}
\mm\ :=\ \frac{1}{\sqrt{2}}(\xx - i\yy)\hspace{.2in},\hspace{.2in}\mb\ :=\ \frac{1}{\sqrt{2}}(\xx + i\yy),
\eeqa
and work with the complex triad $\{\kk,\mm,\mb\}$ in place of $\{\kk,\xx,\yy\}$ (doing this is not necessary, but it allows us to call upon equations already derived in \cite{AA14}).  Observe that $\ip{\mm}{\mm} = \ip{\mb}{\mb} = \ip{\kk}{\mm} = \ip{\kk}{\mm} = 0$, where, e.g., $\ip{\kk}{\mm} = \frac{1}{\sqrt{2}}(\ip{\kk}{\xx} - i\ip{\kk}{\yy})$, while $\ip{\mm}{\mb} = \ip{\kk}{\kk} = 1$.  Since we will need the components of both the Riemann 4-tensor and Ricci tensor with respect to a complex triad $\{\kk,\mm,\mb\}$, we observe here that the latter is given by
$$
\text{Ric}(\cdot,\cdot)\ =\ R(\kk,\cdot,\cdot,\kk) + R(\mm,\cdot,\cdot,\mb) + R(\mb,\cdot,\cdot,\mm).
$$

Next, define the following complex-valued quantities, which comprise the objects of study in the Newman-Penrose formalism:
\beqa
\label{eqn:sc}
\kappa &:=& -\ip{\cd{\kk}{\kk}}{\mm}\hspace{.2in},\hspace{.2in}\rho\ :=\ -\ip{\cd{\mb}{\kk}}{\mm}\hspace{.2in},\hspace{.2in}\sigma\ :=\ -\ip{\cd{\mm}{\kk}}{\mm},\nonumber\\
&&\hspace{.45in}\vep\ :=\ \ip{\cd{\kk}{\mm}}{\mb}\hspace{.2in},\hspace{.2in}\beta\ :=\ \ip{\cd{\mm}{\mm}}{\mb}.\label{eqn:spin}
\eeqa
These so-called \emph{spin coefficients} were first introduced for null vector fields $\kk$ on Lorentzian 4-manifolds in \cite{newpen62}.  The first three spin coefficients in \eqref{eqn:sc} are particularly important, as they encode geometric information regarding the flow of $\kk$.  For one, the flow of $\kk$ is geodesic, by which is meant that $\cd{\kk}{\kk} = 0$, if and only if $\kappa = 0$ (if $\kappa = \vep = 0$, then the vector fields $\xx$ and $\yy$ are parallel along the geodesic flow of $\kk$).  Next, denoting the divergence of $\kk$ by $\text{div}\,\kk$, the real and imaginary parts of the spin coefficient $\rho$ are
\beqa
\label{eqn:rho}
-2\rho\ =\ \text{div}\,\kk + i\,\omega,
\eeqa
where the smooth function $$\omega\ :=\ \ip{\cd{\yy}{\kk}}{\xx} - \ip{\cd{\xx}{\kk}}{\yy}$$ vanishes if and only if the normal subbundle $\kk^{\perp} \subset TM$ is integrable (this follows from Frobenius's theorem, which yields that $\kk^{\perp}$ is integrable if and only if $g(\kk,[\xx,\yy]) = 0$).  Finally, the spin coefficient
\beqa
\label{eqn:sigma}
\sigma\ =\ \frac{1}{2}\Big(\ip{\cd{\yy}{\kk}}{\yy} - \ip{\cd{\xx}{\kk}}{\xx}\Big)\ +\ \frac{i}{2}\Big(\ip{\cd{\yy}{\kk}}{\xx} + \ip{\cd{\xx}{\kk}}{\yy}\Big)
\eeqa
is the (complex) \emph{shear} associated to $\kk$'s flow: its magnitude $|\sigma|$ at any point determines whether an infinitesimal cross section of the flow deforms at that point into an ellipse of the same area.  It is straightforward to verify, for example, that $\kk$ is parallel if and only if $\kappa = \sigma = \rho = 0$.  Certainly the vanishing of $\kappa$ is independent of $\{\xx,\yy\}$; so, too, is the vanishing of $\rho$ and $\sigma$ (see, e.g., \cite[p. 327ff.]{o1995}).  Indeed, $|\sigma|^2$ is the determinant of the trace-free symmetric part of the left-hand matrix in \eqref{eqn:D1} below, while $\omega^2$ is the determinant of its skew-symmetric part.  Therefore both $|\sigma|^2$ and $\omega^2$ are frame-independent smooth functions on $M$.  Note that $\vep = i \ip{\cd{\kk}{\xx}}{\yy}$ is purely imaginary.

\vskip 12pt
Finally, let us rewrite, in terms of a local complex triad $\{\kk,\mm,\mb\}$, the components $R(\kk,\mb,\kk,\mm)$, $R(\kk,\mm,\kk,\mm)$, $R(\mb,\mm,\kk,\mm)$, $R(\kk,\mm,\mm,\mb)$, and $R(\mb,\mm,\mm,\mb)$ of the Riemann 4-tensor in terms of the spin coefficients $\kappa,\rho,\sigma,\vep,\beta$.  The resulting five curvature identities are, respectively,
\beqa
\kk(\rho) - \mb(\kappa) &=& |\kappa|^2 + |\sigma|^2 + \rho^2 + \kappa\bar{\beta} + \frac{1}{2} {\rm Ric}(\kk,\kk),\label{eqn:S1}\\
\kk(\sigma) - \mm(\kappa) &=& \kappa^2 + 2\sigma\vep + \sigma(\rho + \bar{\rho}) - \kappa \beta +{\rm Ric}(\mm,\mm),\phantom{\frac{1}{2}}\label{eqn:S2}\\
\mm(\rho) - \mb(\sigma) &=& 2 \sigma\bar{\beta} + (\bar{\rho}-\rho)\kappa + {\rm Ric}(\kk,\mm),\phantom{\frac{1}{2}}\label{eqn:S3}\\
\kk(\beta) - \mm(\vep) &=& \sigma(\bar{\kappa} - \bar{\beta}) + \kappa (\vep - \bar{\rho}) + \beta(\vep + \bar{\rho}) - {\rm Ric}(\kk,\mm)\phantom{\frac{1}{2}}\label{eqn:S4}\\
\mm(\bar{\beta}) + \mb(\beta) &=& |\sigma|^2 - |\rho|^2 -2|\beta|^2 + (\rho - \bar{\rho})\vep - {\rm Ric}(\mm,\mb) + \frac{1}{2} {\rm Ric}(\kk,\kk).\nonumber\\\label{eqn:S5}
\eeqa
We do the same with the two differential Bianchi identities
\beqa
\ \ (\cd{\kk}{R})(\kk,\mm,\mm,\mb) + (\cd{\mm}{R})(\kk,\mm,\mb,\kk) + (\cd{\mb}{R})(\kk,\mm,\kk,\mm)\ =\ 0,\phantom{\frac{1}{2}}\nonumber\\
\ \ (\cd{\kk}{R})(\mb,\mm,\mm,\mb) + (\cd{\mm}{R})(\mb,\mm,\mb,\kk) + (\cd{\mb}{R})(\mb,\mm,\kk,\mm)\ =\ 0,\nonumber
\eeqa
which then take the forms, respectively,
\beqa
&&\hspace{-.4in}\kk({\rm Ric}(\kk,\mm))\, -\, \frac{1}{2}\mm({\rm Ric}(\kk,\kk))\ +\ \mb({\rm Ric}(\mm,\mm))\ =\ \label{eqn:bid1}\\
&&\hspace{-.1in}\kappa\,{\rm Ric}(\kk,\kk)\ +\ \big(\vep + 2\rho + \bar{\rho}\big){\rm Ric}(\kk,\mm)\ +\ \sigma\,{\rm Ric}(\kk,\mb)\phantom{\frac{1}{2}}\nonumber\\
&&\hspace{1.2in} -\ \big(\bar{\kappa} + 2\bar{\beta}\big){\rm Ric}(\mm,\mm)\, -\, \kappa\,{\rm Ric}(\mm,\mb)\phantom{\frac{1}{2}}\nonumber\\
\text{and}&&\nonumber\\
&&\hspace{-.4in}\mm({\rm Ric}(\kk,\mb))\ +\ \mb({\rm Ric}(\kk,\mm))\, -\, \kk({\rm Ric}(\mm,\mb) - \frac{1}{2}{\rm Ric}(\kk,\kk))\ =\ \label{eqn:bid2}\\
&&\hspace{-.1in}(\rho+\bar{\rho})\big({\rm Ric}(\kk,\kk)-{\rm Ric}(\mm,\mb)\big)\, -\, \bar{\sigma}{\rm Ric}(\mm,\mm)\, -\, \sigma{\rm Ric}(\mb,\mb)\phantom{\frac{1}{2}}\nonumber\\
&&\hspace{1.2in} -\ \big(2\bar{\kappa} + \bar{\beta}\big){\rm Ric}(\kk,\mm)\, -\, \big(2\kappa + \beta\big){\rm Ric}(\kk,\mb).\phantom{\frac{1}{2}}\nonumber
\eeqa
(For a fuller treatment of these derivations, consult \cite{AA14}.)

\section{The signature $(-\!+\!+)$}
\label{sec:thm1}

\begin{proof}[Proof of Theorem \ref{thm:1main}]
Let $g$ be Riemannian metric on a closed 3-manifold $M$ with (globally defined) principal Ricci curvatures $-\mu,f,f$, with $\mu$ a positive number and $f$ a smooth function on $M$ that never takes the values $0,-\mu$.
Consider first the case when $M$ is simply connected.  The smooth bundle endomorphism $\Ric + \mu I\colon TM \lra TM$ has nullity one at every point, hence its kernel $X := \ker{(\Ric + \mu I)}$ is a smooth real line bundle over $M$ (see, e.g., \cite[Theorem 10.34, p. 266]{ISM}).  As $M$ is simply connected, $X$ therefore has a smooth global section $\kk\in\Gamma(X)$ of unit length.
Since $\Ric$ is self-adjoint, it admits a local orthonormal basis of eigenvectors $\{\kk,\vv{u},\vv{v}\}$ such that
\beqa
\label{eqn:ob}
\Ric(\vv{u})\ =\ f \vv{u}\hspace{.2in},\hspace{.2in}\Ric(\vv{v})\ =\ f \vv{v}.
\eeqa
It follows that with respect to the corresponding local complex triad $\{\kk,\vv{n},\overline{\vv{n}}\}$, where $\vv{n} := \frac{1}{\sqrt{2}}(\vv{u} - i\vv{v})$ (and whose corresponding spin coefficients below we denote with a subscript ``$*$"), the components of the Ricci tensor satisfy
\beqa
\label{eqn:ricci}
\text{Ric}(\kk,\vv{n})\ =\ \text{Ric}(\vv{n},\vv{n})\ =\ 0\hspace{.1in},\hspace{.1in}\text{Ric}(\kk,\kk)\ =\ -\mu\hspace{.09in},\hspace{.09in}\text{Ric}(\vv{n},\overline{\vv{n}})\ =\ f.
\eeqa
Let us first assume that $f$ satisfies $\kk(f) = 0$.  Then, inserting \eqref{eqn:ricci} into the differential Bianchi identities \eqref{eqn:bid1} and \eqref{eqn:bid2} yields
$
\kappa_* = \rho_* + \bar{\rho}_* = 0,
$
so that the flow of $\kk$ is geodesic ($\kappa_* = 0$) and divergence-free ($\rho_* + \bar{\rho}_* = 0$).  Inserting these into the curvature identity \eqref{eqn:S1}, its real part simplifies to
\beqa
\label{eqn:ev}
|\sigma_*|^2 - \frac{\omega_*^2}{4}\ =\ \frac{\mu}{2}\cdot
\eeqa
Since $|\sigma_*|^2$ and $\omega_*^2$ are frame-independent, \eqref{eqn:ev} is frame-independent and holds at each point of $M$.  We now replace the vector fields $\vv{u}$ and $\vv{v}$ with global ones more suited to the geometry, as follows.  Consider the normal subbundle $\kk^{\perp} \subset TM$ and the smooth bundle endomorphism
\beqa
\label{eqn:DD}
D\colon \kk^{\perp} \lra \kk^{\perp}\hspace{.2in},\hspace{.2in}Z\ \mapsto\ D(Z)\ :=\ \cd{Z}{\kk}.
\eeqa
(This is well-defined because $\kk$ has unit length.)  In terms of the spin coefficients $\rho_*$ and $\sigma_*$, the matrix of $D$ at a point $p$ is given by
\beqa
\left[\begin{array}{cc}
\ip{\cd{\vv{u}_p}{\kk}}{\vv{u}_p} & \ip{\cd{\vv{v}_p}{\kk}}{\vv{u}_p}\\
\ip{\cd{\vv{u}_p}{\kk}}{\vv{v}_p} & \ip{\cd{\vv{v}_p}{\kk}}{\vv{v}_p}
\end{array}\right]\ =\ 
\left[\begin{array}{cc}
-{\rm re}(\sigma_*) & \frac{\omega_*}{2} + {\rm im}(\sigma_*)\\
-\frac{\omega_*}{2} + {\rm im}(\sigma_*) & {\rm re}(\sigma_*)
\end{array}\right]\bigg|_p\cdot\label{eqn:D1}
\eeqa
By virtue of \eqref{eqn:ev}, each $D_p$ has two distinct eigenvalues $\pm\sqrt{\mu/2}$.  Therefore the two smooth bundle endomorphisms $D \pm \sqrt{\mu/2}\,I\colon \kk^{\perp} \lra \kk^{\perp}$ have nullity one at every point, in which case $X_{\pm} := \ker{(D \pm \sqrt{\mu/2}\,I)}$ admit nowhere vanishing global sections $\tilde{\xx} \in \Gamma(X_-)$ and $\tilde{\yy} \in \Gamma(X_+)$.  Now replace these with the global vector fields
\beqa
\label{eqn:V0}
\xx\ :=\ \tilde{\xx}\hspace{.2in},\hspace{.2in} \yy \ :=\ -\ip{\tilde{\xx}}{\tilde{\yy}}\,\tilde{\xx} + \tilde{\yy},\phantom{\frac{1}{2}}
\eeqa
chosen to have unit length. Writing the matrix $D$ with respect to this new global (orthonormal) frame $\{\kk,\xx,\yy\}$, whose corresponding global complex triad we denote by $\{\kk,\mm,\mb\}$, it follows from \eqref{eqn:D1} and the identity $D\xx = \sqrt{\mu/2}\,\xx$ that its spin coefficient $\sigma$ (not to be confused with $\sigma_*$ above, though of course $|\sigma_*|^2 = |\sigma|^2$) satisfies $\text{re}(\sigma) = -\sqrt{\mu/2}$ and $\text{im}(\sigma) = \omega/2$.  To summarize, then, we have shown that there exists a global complex triad $\{\kk,\mm,\mb\}$ whose corresponding spin coefficients $\kappa$, $\rho$, and $\sigma$ satisfy
\beqa
\label{eqn:rs1}
\kappa\ =\ 0\hspace{.2in},\hspace{.2in}\rho\ =\ -i\,\frac{\omega}{2}\hspace{.2in},\hspace{.2in}\sigma\ =\ -\sqrt{\frac{\mu}{2}} + i\,\frac{\omega}{2}\cdot
\eeqa
(Being geodesic ($\kappa = 0$) and divergence-free ($\rho + \bar{\rho} = 0$) is, of course, independent of the complex triad used.)  The virtue of this particular complex triad is the form of its shear $\sigma$ in \eqref{eqn:rs1}, which nicely simplifies the curvature identities \eqref{eqn:S1}, \eqref{eqn:S2}, \eqref{eqn:S3}, and \eqref{eqn:S5} above.  Indeed, inserting \eqref{eqn:rs1} into the imaginary part of \eqref{eqn:S1} yields $\kk(\omega) = 0$, which immediately implies that $\kk(\sigma) = 0$.  This in turn simplifies \eqref{eqn:S2} to $2\sigma\vep = -\text{Ric}(\mm,\mm)$, which yields $\omega\, \vep = 0$.  Next, inserting \eqref{eqn:rs1} into \eqref{eqn:S3} yields real and imaginary parts
\beqa
\label{eqn:xw}
\xx(\omega)\ =\ 2\sqrt{\frac{\mu}{2}}\,\text{div}\,\yy - \omega\,\text{div}\,\xx\hspace{.2in},\hspace{.2in}\sqrt{\frac{\mu}{2}}\,\text{div}\,\xx + \frac{\omega}{2}\,\text{div}\,\yy\ =\ 0,
\eeqa
where $\beta = \frac{1}{\sqrt{2}}\left(\ip{\cd{\yy}{\xx}}{\yy} + i \ip{\cd{\xx}{\xx}}{\yy}\right) = \frac{1}{\sqrt{2}} (\text{div}\,\xx - i\,\text{div}\,\yy)$ (the latter because $\cd{\kk}{\kk} = 0$).  Finally, \eqref{eqn:S5} simplifies via \eqref{eqn:ev} and $\omega\,\vep = 0$ to
$$
\xx(\text{div}\,\xx) + \yy(\text{div}\,\yy)\ =\ -(\text{div}\,\xx)^2 - (\text{div}\,\yy)^2 - f,\phantom{\frac{1}{2}}
$$
which in turn further simplifies, via \eqref{eqn:xw}, to
\beqa
\label{eqn:V}
\bigg(\!\!-\!\frac{\omega}{\sqrt{2\mu}}\,\xx + \yy\bigg)(\text{div}\,\yy)\ =\ -f.
\eeqa
But this is impossible on a closed manifold with $f$ nowhere vanishing.  This completes the proof in the case that $M$ is simply connected.  If $M$ is not simply connected, then pass to its universal cover $\pi\colon (\widetilde{M},\tilde{g}) \lra (M,g)$, which has principal Ricci curvatures $-\mu,f \circ \pi,f \circ \pi$.  Repeating our argument on $(\widetilde{M},\tilde{g})$ with corresponding global orthonormal basis of eigenvectors $\{\vv{K},\vv{X},\vv{Y}\}$, we once again arrive at \eqref{eqn:V}.  Although $(\widetilde{M},\tilde{g})$ need not be compact, we still obtain a contradiction, given that $\text{div} \vv{Y}$ and $f \circ \pi$ are bounded in $\widetilde{M}$, and the vector field in \eqref{eqn:V} is complete.  The reason is because $d\pi(\vv{Y})$ is equal to $\yy$ in \eqref{eqn:V0} up to sign; but as $|\text{div}\,\yy|$ is a continuous function on $M$, it is bounded (observe that $\yy$ in \eqref{eqn:V0}, though in general defined only locally when $M$ is not simply connected, is nonetheless unique up to sign).  Hence $\text{div} \vv{Y}$ is bounded in $\widetilde{M}$, which contradicts \eqref{eqn:V} on $\widetilde{M}$ (written with respect to $\{\vv{K},\vv{X},\vv{Y}\}$).

\vskip 12pt
Now suppose that $f$ is a smooth function on $M$ that never takes on the values $0,-\mu$, but that it does not necessarily satisfy $\kk(f) = 0$.  Recalling $D$ in \eqref{eqn:DD} and \eqref{eqn:D1} above (with $\text{div}\,\kk$ reinstated), begin by defining the function
\beqa
\label{eqn:H}
H\ :=\ 
\text{det}\,D - \frac{\mu}{2}
\ =\ 
\frac{\omega^2}{4} - |\sigma|^2 + \frac{(\text{div}\,\kk)^2}{4} - \frac{\mu}{2} \ .
\eeqa
Even though $f$ is no longer assumed to be constant, observe that the first differential Bianchi identity \eqref{eqn:bid1} nonetheless yields $\kappa = 0$, so that $\kk$ still has geodesic flow (the second differential Bianchi identity \eqref{eqn:bid2} now yields $\kk(f) = -(\text{div}\,\kk)(\mu+f)$, to which we will return later).  With $\kappa = 0$, the real and imaginary parts of \eqref{eqn:S1} are, respectively,
\beqa
\label{eqn:S1c}
\kk(\text{div}\,\kk)\ =\ 2H - (\text{div}\,\kk)^2 + 2\mu\hspace{.2in},\hspace{.2in}\kk(\omega)\ =\ -(\text{div}\,\kk)\,\omega,
\eeqa
while \eqref{eqn:S2}, via the identity $\vep + \bar{\vep} = 0$ and the fact that $\text{Ric}(\mm,\mm) = 0$ in a local complex triad satisfying \eqref{eqn:ob} (with $f$ in place of $\lambda$), implies
$
\kk(|\sigma|^2) = -2(\text{div}\,\kk)\,|\sigma|^2.
$
(Since $|\sigma|^2$ is frame-independent, so is this equation.)  These three equations combine to yield the following evolution equation for $H$ along the flow of $\kk$:
\beqa
\label{eqn:detH}
\kk(H)\ =\ -(\text{div}\,\kk)\,H.
\eeqa
This equation implies that along any integral curve $\gamma(s)$ of $\kk$, either $H \circ \gamma$ is nowhere zero or else it vanishes identically.  
We claim that $H$ cannot vanish; for if it does, inserting $\theta(s) := (\Div\kk\circ\gamma)(s)$ into equation \eqref{eqn:S1c} yields
\beqa
\label{eqn:ray3}
\frac{d\theta}{ds}\ =\ -\theta^2 + 2\mu\hspace{.2in}\forall s \in \RR.
\eeqa
The complete (non-singular) solutions to \eqref{eqn:ray3} are the constant solutions $\theta(s) = \pm\sqrt{2\mu}$ and (up to a constant shift of $s$) 
\beqa
\label{eqn:tanh}
\theta(s)\ =\ \sqrt{2\mu}\tanh(\sqrt{2\mu}\,s) .
\eeqa
The second differential Bianchi identity \eqref{eqn:bid2} implies that $g(s) := (\mu+f)\circ\gamma$ satisfies $g'=-\theta\, g$, and inserting \eqref{eqn:tanh} yields the general solution 
\beqa
g(s)\ =\ g_0 / \cosh(\sqrt{2\mu}s) \,,\nonumber
\eeqa
which contradicts $g>\mu$.  Similarly, the constant solutions $\theta(s) = \pm\sqrt{2\mu}$ yield $g(s) = g_0\,e^{\mp \sqrt{2\mu}\,s}$, and hence also contradict $g > \mu$.
Thus $H$ vanishes nowhere on $M$.
Next, consider the function $\ell(s) := ((1/H)\circ \gamma)(s)$; then \eqref{eqn:detH} and the first equation in \eqref{eqn:S1c} together yield
$$
\frac{d^2\ell}{ds^2}\ =\ 2 +2\mu\,\ell\hspace{.2in}\forall s \in \RR,
$$
whose general solution is $\ell(s) = -1/\mu + c_1\,e^{\sqrt{2\mu}\,s} + c_2\,e^{-\sqrt{2\mu}\,s}$.  
This means that either $\ell$ is constant along $\gamma$, or diverges as $s$ goes to (at least one of) $\pm\infty$.
Consider the latter case.
It is straightforward to show that $\ell(s)$ satisfies $\ell'=\theta\,\ell$, which together with $g'=-\theta g$ implies that $g=c/\ell$ for some constant $c$.
But this contradicts $g>\mu$, so we conclude that $H$ must be a nonzero constant along all integral curves of $\kk$.
Then by \eqref{eqn:detH} $\Div\kk=0$, which implies that $\kk(f)=0$; but the first part of the proof showed that this is impossible.
\end{proof}

\section{The signature $(0\!+\!+)$}
\label{sec:thm}
\begin{proof}[Proof of Theorem \ref{thm:main}]
The principal Ricci curvatures of $(M,g)$ are $0,f,f$, with $f$ not assumed to be constant.  For $\kk$ as given and any local complex triad $\{\kk,\mm,\mb\}$, the Ricci tensor satisfies
\beqa
\label{eqn:leq}
\text{Ric}(\kk,\kk)\ =\ \text{Ric}(\kk,\mm)\ =\ \text{Ric}(\mm,\mm)\ =\ 0\hspace{.15in},\hspace{.15in}\text{Ric}(\mm,\mb)\ =\ \frac{S}{2},
\eeqa
where $S$ is the scalar curvature of $(M,g)$ which, by assumption, is positive and satisfies $\kk(S) = 0$.  We now show that $\kk$ is parallel: $\kappa = \rho = \sigma = 0$.  Inserting \eqref{eqn:leq} into the differential Bianchi identities \eqref{eqn:bid1} and \eqref{eqn:bid2} yields $\kappa = (\text{div}\,\kk) = 0$ (we remark here in passing that these also follow from the contracted Bianchi identity).  These in turn simplify the real part of \eqref{eqn:S1} to
$$
\frac{\omega^2}{4} - |\sigma|^2\ =\ 0.
$$
Now we proceed as in the proof of Thereom \ref{thm:2main}: $D$ will have zero determinant everywhere, with matrix given by
\beqa
\label{eqn:matrix1}
\left[\begin{array}{cc}
0 & \omega(p)\\
0 & 0
\end{array}\right]\cdot\nonumber
\eeqa
We will thus look to work in a frame satisfying 
\beqa
\label{eqn:sw}
\sigma\ =\ i\,\frac{\omega}{2}\ =\ \bar{\rho},
\eeqa
the difference here being that we do not know that $\omega$ is nowhere vanishing.  Now, if $\omega = 0$, then $\rho = \sigma = 0$ and we are done.
Our strategy is thus to show that the open subset
$$
\uu \ :=\ \{p \in M : \omega(p) \neq 0\} 
$$
is empty.
Clearly it suffices to prove this for each connected component, so we may assume $\uu$ is connected.  We first consider the case where $\uu$ is simply connected.  The map $D$ has constant rank 1 in $\uu$, hence $X = \ker{D}\vert_{\uu}$ and its orthogonal complement in $\kk^\perp\vert_{\uu}$, $Y$, are smooth real line bundles over $\uu$.  As $\uu$ is simply connected they have smooth global sections $\xx\in\Gamma(X)$ and $\yy\in\Gamma(Y)$ of unit length, which together with $\kk$ form an orthonormal frame on $T\uu$ satisfying
$
\cd{\xx}{\kk} = 0.
$
With respect to this frame, the quantities $\rho$ and $\sigma$ satisfy \eqref{eqn:sw}.
Using $\kappa = \text{Ric}(\kk,\mm) = 0$, \eqref{eqn:S3} simplifies to $\xx(\omega) = -\sqrt{2}\,\omega\,\bar{\beta}$, whose real and imaginary parts of are, respectively,
\beqa
\label{eqn:x22}
\xx(\omega)\ =\ -(\text{div}\,\xx)\,\omega\hspace{.2in},\hspace{.2in}\ip{\cd{\xx}{\xx}}{\yy}\,\omega\ =\ 0.\phantom{\frac{1}{2}}
\eeqa
Observe that because $\omega$ is nowhere vanishing in $\uu$, the second equation in \eqref{eqn:x22} gives $\cd{\xx}{\xx} = 0$, so that the flow of $\xx$ is everywhere geodesic. 
The imaginary part of \eqref{eqn:S1} yields a similar equation for $\kk$, $\kk(\omega) = -(\text{div}\,\kk)\,\omega = 0$, and since $\kk(\sigma) = (i/2)\,\kk(\omega)$, \eqref{eqn:S2} reduces to $\omega\,\vep = 0$, hence $\vep = 0$.
Finally, substituting $\beta = \bar{\beta} = \frac{1}{\sqrt{2}}\text{div}\,\xx, \vep = 0$, and $|\sigma|^2 = |\rho|^2$ into \eqref{eqn:S5} yields simply
\beqa
\label{eqn:rayx1}
\xx(\text{div}\,\xx)\ =\ -(\text{div}\,\xx)^2 - \frac{S}{2}\cdot
\eeqa
(Compare \eqref{eqn:rayx1} with \eqref{eqn:rayx1a} in the proof of Theorem \ref{thm:2main} above.)  Now let $\gamma\colon [0,b) \lra \uu$ be an integral curve of $\xx$ that is maximally extended to the right.  We claim that $b=\infty$.  
Indeed, suppose $b$ is finite.
Because $\cd{\xx}{\xx}=0$, $\gamma$ is a geodesic, and thus right-extendible in $M$ by completeness.
It follows that for $b$ finite the limit $\lim_{s\to b}\gamma(s)$ exists (in $M$), and is not in $\uu$.
Setting $\theta(s) := (\text{div}\,\xx\circ \gamma)(s)$ and $\theta_0 := \theta(0)$, the first equation in \eqref{eqn:x22} yields 
\beqa
\label{eqn:wbd}
(\omega \circ \gamma)(s)\ =\ \omega_0\,e^{-\!\int_0^s \theta(u)\,du}\hspace{.2in}\forall s \in [0,b), 
\eeqa
where, without loss of generality, $\omega_0 := \omega(\gamma(0))$ can be chosen to be positive by an appropriate choice of $\yy$.  
By \eqref{eqn:rayx1}, $\theta(s)$ is strictly decreasing (recall that $S$ is positive), so that $\theta(s) < \theta_0$ for all $s \in (0,b)$.  Hence \eqref{eqn:wbd} satisfies
\beqa
\label{eqn:wlb}
(\omega \circ \gamma)(s)\ >\ \omega_0\,e^{-\theta_0s}\hspace{.2in}\forall s \in (0,b) .
\eeqa

As $\uu = \{p \in M : \omega(p) \neq 0\}$ does not contain $\lim_{s\to b}\gamma(s)$, we conclude that $\lim_{s\to b}(\omega\circ\gamma)(s)=0$.
This contradicts \eqref{eqn:wlb}; hence $b$ must be infinite.  
Repeating the argument for $-\xx$ implies that the flow of $\xx$ is complete in $\uu$.  With that established, let $\gamma$ be an integral curve of $\xx$ and set $h(s) := 1/(\omega \circ \gamma)(s)$.  Then \eqref{eqn:rayx1} and the first equation in \eqref{eqn:x22} combine to yield
$$
\frac{d^2 h}{ds^2}\ =\ -\frac{S(\gamma(s))}{2}h(s)\ <\ 0\hspace{.2in}\forall s \in \RR,
$$
where the inequality is due to $S > 0$ and $h > 0$.  
Any positive function $h$ on $\RR$ with $h''<0$ lies below its tangent lines, so that it must either cross zero at some point, or be everywhere constant, which is impossible as $h'' < 0$.   The only way to avoid this contradiction is for $\uu$ to be empty.  

\vskip 12pt

We now return to the general case when $\uu$ is connected but not simply connected.
Lifting to the Riemannian universal covering space $\pi\colon\tuu\lra\uu$, the lifts $\widetilde{X}$ and $\widetilde{Y}$ of $X$ and $Y$ have global unit length sections $\tilde{\xx}$ and $\tilde{\yy}$, and $\kk$ lifts to a unit length vector field $\tilde\kk$.
The argument proceeds as before: any integral curve of $\tilde\xx$ is a geodesic $\tilde\gamma$ in $\tuu$, which projects to a geodesic $\gamma=\pi\circ\tilde\gamma$ in $\uu$.
Completeness of $M$ applied to $\gamma$ once again implies that the flow of $\tilde\xx$ is complete in $\tuu$.
The argument of the previous paragraph then gives a contradiction unless $\tuu$ is empty.  This completes the proof that $\kk$ is parallel.
Finally, observe that the universal cover of $(M,g)$ splits isometrically as $\RR \times \widetilde{N}$; this follows from the de Rham decomposition theorem (see, e.g., \cite[Theorem 56, p. 253]{petersen} and also \cite{schmidt14}).  Note that when the scalar curvature is constant, then $\widetilde{N} = S^2$.
\end{proof}

\section{The signature $(0\!+\!-)$}
\label{sec:thm2}

\begin{proof}[Proof of Theorem \ref{thm:2main}]
Consider first when $(M,g)$ is complete with globally constant principal Ricci curvatures $0,\lambda,-\lambda$; we then assume the existence of a unit length vector field $\kk \in \mathfrak{X}(M)$ with geodesic flow ($\cd{\kk}{\kk} = 0$) and satisfying $\Ric(\kk) = 0$.  Let $\{\kk,\xx,\yy\}$ be a local orthonormal basis of eigenvectors of $\Ric$ such that
\beqa
\label{eqn:sfricci}
\Ric(\xx)\ =\ \lambda\xx\hspace{.2in},\hspace{.2in}\Ric(\yy)\ =\ -\lambda\yy.\nonumber
\eeqa
If $\{\kk,\mm,\mb\}$ is the corresponding local complex triad, then $\text{Ric}(\mm,\mm) = \lambda$ and all other components are zero (with respect to $\{\kk,\mm,\mb\}$).  Accordingly, the differential Bianchi identities \eqref{eqn:bid1} and \eqref{eqn:bid2} yield, respectively,
\beqa
\label{eqn:bidbid}
\beta\ =\ 0\hspace{.2in},\hspace{.2in}\sigma + \bar{\sigma}\ =\ 0,
\eeqa
while the curvature identities \eqref{eqn:S1} and \eqref{eqn:S2} reduce to 
\beqa
\kk(\rho) &=& |\sigma|^2 + \rho^2,\label{eqn:S1b}\phantom{\frac{1}{2}}\\
\kk(\sigma)&=& 2\sigma\vep + \sigma(\rho + \bar{\rho}) + \lambda.\phantom{\frac{1}{2}}\label{eqn:S2b}
\eeqa
Since $\beta = 0$ and $\sigma$ is imaginary, an additional equation is provided via \eqref{eqn:S5}, which reads
\beqa
\label{eqn:S5b}
|\sigma|^2 - |\rho|^2 - i\,\omega \vep\ =\ 0.
\eeqa
Now suppose that $\omega$ is zero at a point $p$ in the domain of $\{\kk,\mm,\mb\}$; since the imaginary part of \eqref{eqn:S1b} yields $\kk(\omega) = -(\text{div}\,\kk)\,\omega$ as usual, $\omega$ must vanish along the (complete) integral curve $\gamma(s)$ of $\kk$ through $p$.  Then by \eqref{eqn:S5b}, $|\sigma|^2 = |\rho|^2 = (1/4)(\text{div}\,\kk)^2$ everywhere along $\gamma$ (recall that both $|\sigma|^2$ and $|\rho|^2$ are frame-independent), so that the real part of \eqref{eqn:S1b} gives
\beqa
\label{eqn:Raylast}
\frac{d\theta}{ds}\ =\ -\theta^2,
\eeqa
where, as before, we have set $\theta(s) = (\text{div}\,\kk \circ \gamma)(s)$; because $\kk$ has complete flow, $\theta$ is defined for all $s \in \RR$.  Then \eqref{eqn:Raylast} implies that $\theta$ is either identically zero or else strictly positive, but in fact the second case cannot occur: if $\theta > 0$, then $1/\theta(s) = s + b$, a contradiction.  On the other hand, if $\theta(s) \equiv 0$, then $\lambda = 0$ by \eqref{eqn:S2b}, a contradiction once again.  Thus we conclude that in fact $\omega$ must be nowhere vanishing on $M$; i.e., that
$
\{p \in M : \omega(p) \neq 0\} = M.
$
(We mention in passing that this is equivalent to the 1-form $g(\kk,\cdot)$ being a contact form on $M$; see below.)  Next, observe that
\beqa
\label{eqn:two2}
\kk(|\sigma|^2)\ =\ -2(\text{div}\,\kk)|\sigma|^2\hspace{.2in},\hspace{.2in}\kk(|\rho|^2)\ =\ -(\text{div}\,\kk) (|\sigma|^2 + |\rho|^2),\phantom{\frac{1}{2}}
\eeqa
the former via \eqref{eqn:bidbid}, \eqref{eqn:S2b} and $\vep + \bar{\vep} = 0$, the latter via \eqref{eqn:S1b} (recall that $|\sigma|^2$ and $|\rho|^2$ are frame-independent; also, note that these equations can be defined everywhere along a given (geodesic) integral curve $\gamma(s)$ of $\kk$, by parallel translating two orthonormal vectors $x,y \in \kk_{\gamma(0)}^{\perp}$ along $\gamma$ and writing down \eqref{eqn:two2} with respect to their parallel translates).  Armed with these, as well as with $\kk(\omega) = -(\text{div}\,\kk)\,\omega$, the derivative of \eqref{eqn:S5b} along $\kk$ simplifies to give
$
-i\,\omega \kk(\vep) = 0,
$
hence that $\kk(\vep) = 0$ everywhere in the domain of $\{\kk,\mm,\mb\}$.  In fact the spin coefficient $\vep = i \ip{\cd{\kk}{\xx}}{\yy}$ is a constant here.  This follows from setting $\beta = \kappa = \text{Ric}(\kk,\mm) = 0$ in \eqref{eqn:S4}, to obtain $\mm(\vep) = 0$, and hence that $\vep$ is a constant: $\vep = i c$.  The significance of this fact is seen by taking the real part of \eqref{eqn:S2b} to obtain
\beqa
\label{eqn:vep}
-i c\,(\sigma - \bar{\sigma})\ =\ \lambda 
\eeqa
where we have used the fact that $\sigma$ is imaginary.  In other words, the shear $\sigma$ is also a constant in the domain of $\{\kk,\mm,\mb\}$; since $|\sigma|^2$ is frame-independent, it follows that $|\sigma|^2$ is a global constant on $M$.  If this constant is zero, then by \eqref{eqn:S2b} we have $\lambda = 0$, a contradiction.  Since it is nonzero, the first equation in \eqref{eqn:two2} dictates that $\text{div}\,\kk = 0$,
so that $\rho + \bar{\rho} = 0$.  With this established, the real part of \eqref{eqn:S1b} now gives
$
|\sigma|^2 + \rho^2\ =\ |\sigma|^2 - |\rho|^2 = 0.
$
But this combines with \eqref{eqn:S5b} to yield $\omega\,\vep = 0$, hence that $\vep = i c = 0$, hence that $\lambda = 0$ by \eqref{eqn:vep}, a contradiction once again.  Thus $\lambda$ cannot be constant.
\vskip 12pt
Now suppose that $M$ is closed and that $\kk$, in addition to having geodesic flow, is also divergence-free, but relax the condition that the nonzero principal Ricci curvatures are constants.  For clarity, let us write them now as $\pm f$, where $f > 0$ is a smooth function on $M$.  Once again, observe that with respect to any local complex triad $\{\kk,\mm,\mb\}$,
\beqa
\label{eqn:simple}
\kappa\ =\ \rho + \bar{\rho}\ =\ \text{Ric}(\kk,\cdot)\ =\ \text{Ric}(\mm,\mb)\ =\ 0.
\eeqa
Set $\psi := \text{Ric}(\mm,\mm)$; then $\psi$ is nowhere vanishing (because $f$ is so) and also the only nonzero component of the Ricci tensor with respect to $\{\kk,\mm,\mb\}$.  Inserting \eqref{eqn:simple} into the second differential Bianchi identity \eqref{eqn:bid2}, as well as into the curvature identities \eqref{eqn:S1} and \eqref{eqn:S2}, yields, respectively,
\beqa
\sigma \bar{\psi} + \bar{\sigma} \psi &=& 0,\label{eqn:bid2a}\phantom{\frac{1}{2}}\nonumber\\
\kk(\rho) &=& |\sigma|^2 + \rho^2,\label{eqn:S1a}\phantom{\frac{1}{2}}\\
\kk(\sigma)&=& 2\sigma\vep + \psi.\phantom{\frac{1}{2}}\label{eqn:S2a}
\eeqa
Then just as with the first equation in \eqref{eqn:two2}, this time with $\text{div}\,\kk = 0$,
\beqa
\label{eqn:magshear}
\kk(|\sigma|^2)\ =\ 0.
\eeqa
Along a given integral curve $\gamma(s)$ of $\kk$, set $h(s) := (|\sigma|^2 \circ \gamma)(s)$.  If $h(0) = 0$, then by \eqref{eqn:magshear} $h$ is everywhere zero, in which case $\psi \circ \gamma = 0$ by \eqref{eqn:S2a}, contradicting the fact that it is nowhere vanishing.  Thus $h > 0$ along $\gamma$ and so $|\sigma|^2 > 0$ on $M$; because the real part of \eqref{eqn:S1a} is $|\sigma|^2 - \omega^2/4 = 0$, it follows that the global smooth function $\omega^2$ is nowhere vanishing on $M$.  Consider now the determinant $H$ of the bundle endomorphism $D$ given by \eqref{eqn:DD} above; a computation shows that $H = \omega^2/4 - |\sigma|^2 + (\text{div}\,\kk)^2/4 = 0$.  Because $\omega$ is nowhere vanishing, it follows that $D$ has nullity one at every point of $M$.  Assume that $M$ is simply connected; then the smooth line bundle $X := \ker{D}$ has a global unit section $\xx \in \Gamma(X)$, which satisfies
$
\cd{\xx}{\kk} = 0.
$
Similarly, the orthogonal complement of $X$ in $\kk^{\perp}$, $Y$, is a smooth line bundle, so it, too, has a global unit section $\yy \in \Gamma(Y)$.  Let $\{\kk,\mm,\mb\}$ denote the complex triad associated to the global orthonormal frame $\{\kk,\xx,\yy\}$.  Then its spin coefficients $\rho$ and $\sigma$ in particular satisfy
$
\sigma = i\,\frac{\omega}{2} = \bar{\rho}
$
(this follows from \eqref{eqn:rho}, \eqref{eqn:sigma}, and the fact that $\text{div}\,\kk = 0$).  Note that $\omega$ is now a nowhere vanishing smooth function globally defined on $M$; we can, by considering $-\xx$ if necessary, assume that $\omega > 0$.  
Next, using $\kappa = \text{Ric}(\kk,\mm) = 0$, \eqref{eqn:S3} simplifies to $\xx(\omega) = -\sqrt{2}\,\omega\,\bar{\beta}$,
whose real and imaginary parts are, respectively,
\beqa
\label{eqn:x22a}
\xx(\omega)\ =\ -(\text{div}\,\xx)\,\omega\hspace{.2in},\hspace{.2in}\ip{\cd{\xx}{\xx}}{\yy}\,\omega\ =\ 0.\phantom{\frac{1}{2}}
\eeqa
Observe that because $\omega$ is nowhere vanishing, it follows from the second equation in \eqref{eqn:x22a} that $\cd{\xx}{\xx} = 0$, so that the flow of $\xx$, like that of $\kk$, is everywhere geodesic (furthermore, $\beta$ is real). 
The imaginary part of \eqref{eqn:S1a} yields $\kk(\omega) = 0$, hence $\kk(\sigma) = (i/2)\,\kk(\omega) = 0$, so that \eqref{eqn:S2a} yields $\psi = -i\omega\,\vep$; it follows that $\psi = \text{Ric}(\mm,\mm)$ is real in the frame $\{\kk,\mm,\mb\}$.
Finally, consider \eqref{eqn:S5}.  
Substituting $\beta = \bar{\beta} = \frac{1}{\sqrt{2}}\text{div}\,\xx, \psi = -i\,\omega\,\vep$, and $|\sigma|^2 = |\rho|^2$ into \eqref{eqn:S5} yields simply
\beqa
\label{eqn:rayx1a}
\xx(\text{div}\,\xx)\ =\ -(\text{div}\,\xx)^2 + \psi.
\eeqa
Let $\gamma$ be an integral curve of $\xx$ and set $\ell(s) := 1/(\omega \circ \gamma)(s) > 0$.  Then \eqref{eqn:rayx1a} and the first equation in \eqref{eqn:x22a} combine to yield
\beqa
\label{eqn:ell}
\frac{d^2 \ell}{ds^2}\ =\ \psi(\gamma(s))\,\ell(s)\hspace{.2in}\forall s \in \RR.
\eeqa
This implies that $\psi \circ \gamma$ must be positive, for otherwise $\ell(s) > 0$ is incompatible with $\ell''(s) < 0$ (note that $\psi$ can never be zero at any point, for then so would $f$).  Since $\text{Ric}(\mm,\mb) = 0$ implies $\text{Ric}(\yy,\yy) = -\text{Ric}(\xx,\xx)$, observe that $\text{Ric}(\xx,\xx) = \text{Ric}(\mm,\mm) = \psi > 0$ on $M$; since the latter is closed, it follows that $\text{Ric}(\xx,\xx) \geq b$ for some positive constant $b$.  The significance of this can be seen by considering the analogue of the curvature identity \eqref{eqn:S1} \emph{for $\xx$} (rather than $\kk$).  In other words, for the complex triad $\{\xx,\nn,\nb\}$, with $\nn := \frac{1}{\sqrt{2}}(\kk-i \yy)$ (and whose corresponding spin coefficients we distinguish with a superscript ``${\sim}$", noting that $\tilde{\kappa} = 0$ because $\xx$ has geodesic flow), one obtains the pair of equations
$$
\xx(\text{div}\,\xx)\ =\ \frac{\tilde{\omega}^2}{2} - 2|\tilde{\sigma}|^2 - \frac{(\text{div}\,\xx)^2}{2} - \text{Ric}(\xx,\xx)\hspace{.15in},\hspace{.15in}\xx(\tilde{\omega})\ =\ -(\text{div}\,\xx)\,\tilde{\omega}.
$$
But because $\text{Ric}(\xx,\xx) \geq b$, these equations together imply that $\tilde{\omega}$ is nowhere vanishing on $M$, hence that $g(\xx,\cdot)$ is a contact form, because
$$
dg(\xx,\cdot)(\kk,\yy)\ =\ -\tilde{\omega}.
$$
Furthermore, since $\xx$ has unit length and geodesic flow, it is the Reeb vector field of $g(\xx,\cdot)$ (i.e., the unique vector field satisfying $g(\xx,\xx) = 1$ and $\xx \lrcorner\,dg(\xx,\cdot) = 0$).  By the Weinstein conjecture in dimension three \cite{taubes07}, it follows that $\xx$ has a \emph{closed} integral curve $\gamma(s)$.  But on closed $\gamma$ we cannot everywhere have $\ell''(s)>0$, in contradiction with \eqref{eqn:ell}.  This completes the proof when $M$ is simply connected.

\vskip 12pt
If $M$ is not simply connected, then pass to the finite-sheeted cover $\pi\colon (\widetilde{F},\tilde{g}) \lra (M,g)$ trivializing the line bundles $X$ and $Y$, which is compact with principal Ricci curvatures $0,f \circ \pi, -f \circ \pi$.  Repeating our argument on $(\widetilde{F},\tilde{g})$, the proof is complete.
\end{proof}

\section{Ricci solitons and the Newman-Penrose Formalism}
\begin{proof}[Proof of Lemma \ref{lemma:Ricci}]
In fact this result is purely local: we need only stipulate
that for any $p \in M$, there be a neighborhood $\uu$ of $p$ and a unit length vector field $\kk$ defined on $\uu$ such that $(\uu,g,\kk)$ satisfies \eqref{eqn:rs}.
Let $\{\kk,\mm,\mb\}$ be a local complex tetrad about $p$.
Then a computation shows that the components of the Ricci tensor with respect to $\{\kk,\mm,\mb\}$ satisfy
\beqa
\text{Ric}(\kk,\kk) &=& \frac{\lambda}{2},\nonumber\\
\text{Ric}(\kk,\mm) &=& \frac{\kappa}{2},\nonumber\\
\text{Ric}(\mm,\mb) &=& \frac{\rho +\bar{\rho}}{2} + \frac{\lambda}{2},\nonumber\\
\text{Ric}(\mm,\mm) &=& \sigma.\nonumber
\eeqa
Along with \eqref{eqn:S1} (and its complex conjugate), inserting these identities into the second differential Bianchi identity \eqref{eqn:bid2} yields
$
|\kappa|^2 + |\rho|^2 + |\sigma|^2 = \frac{\lambda}{4},
$
so that $\lambda \geq 0$.  If $\lambda = 0$, then $\kappa = \rho = \sigma = 0$ if and only if $\kk$ is parallel, hence in particular $\text{Ric} = \mathscr{L}_\kk\,g = 0$.  Thus if the Ricci soliton is nontrivial, it must satisfy $\lambda > 0$ in $\uu$. 
\end{proof}

\begin{proof}[Proof of Lemma \ref{lemma:Ricci2}]
Given a nowhere vanishing $\tilde{\kk}$ in the kernel of the Ricci transformation for which $(M,g,\tilde{\kk})$ is a Ricci soliton, set $\tilde{\kk}/|\tilde{\kk}| := \kk$ and let $\{\kk,\mm,\mb\}$ be a local complex tetrad.  Then $\text{Ric}(\kk,\cdot) = \text{Ric}(\mm,\mb) = 0$,
and a computation shows that the components of the Ricci tensor with respect to $\{\kk,\mm,\mb\}$ now satisfy
\beqa
0\ =\ \text{Ric}(\kk,\kk) &=& \frac{\lambda}{2} - \kk(|\tilde{\kk}|),\label{eqn:lambda5}\\
0\ =\ \text{Ric}(\kk,\mm) &=& -\frac{1}{2} (\mm(|\tilde{\kk}|) - |\tilde{\kk}| \kappa),\nonumber\\
\frac{S}{2}\ =\ \text{Ric}(\mm,\mb) &=& \frac{1}{2}(\lambda - |\tilde{\kk}| (\text{div}\,\kk)),\label{eqn:div5}\\
\text{Ric}(\mm,\mm) &=& |\tilde{\kk}| \sigma\cdot\label{eqn:sigma5}
\eeqa
Inserting these into the second differential Bianchi identity \eqref{eqn:bid2} gives
\beqa
\label{eqn:sc5}
|\tilde{\kk}| |\sigma|^2\ =\ (\text{div}\,\kk)\frac{S}{4}\cdot
\eeqa
(Constant scalar curvature is used here.)  Combined with \eqref{eqn:div5}, this gives
\beqa
\label{eqn:lambdaS}
\frac{(\lambda - S)S}{4}\ =\ |\tilde{\kk}|^2 |\sigma|^2\ \geq\ 0.\nonumber
\eeqa
Since $|\tilde{\kk}| (\text{div}\,\kk)$ is a constant, its derivative along $\kk$ yields, via \eqref{eqn:lambda5}
and \eqref{eqn:div5},
$$
\kk(\text{div}\,\kk)\ =\ -\frac{\lambda(\lambda - S)}{4|\tilde{\kk}|^2}\cdot
$$
If $\lambda \neq S$, then $\kk(\text{div}\,\kk)$ is bounded away from zero (via $|\tilde{\kk}|$), a contradiction on compact $M$.  Thus we must have $\lambda = S$.  If $S = \lambda = 0$, then $\sigma = 0$ by \eqref{eqn:sc5}; combined with \eqref{eqn:sigma5}, the manifold is therefore flat.
\end{proof}

\begin{proof}[Proof of Corollary \ref{lemma:main3}]
By assumption, there is a unit length vector field $\kk$ satisfying $R(\kk,\cdot,\cdot,\cdot) = 0$, hence $R(f\kk,\cdot,\cdot,\cdot) = 0$ for any smooth function $f$ on $M$.  Because $(M,g)$ is connected, complete, and has constant positive scalar curvature, it follows from Theorem \ref{thm:main} that $\nabla \kk = 0$.  In particular, for any local complex tetrad $\{\kk,\mm,\mb\}$, its spin coefficients satisfy $\kappa = \rho = \sigma = 0$ and \eqref{eqn:leq}.
If $(M,g,f\kk)$ is to be a Ricci soliton for a suitably chosen $f$ (which need not be nowhere vanishing), then in the domain of $\{\kk,\mm,\mb\}$ the 2-tensor $(\lambda/2)g - (1/2)\mathscr{L}_{f\kk}\,g$
must equal the Ricci tensor, hence we must show that the following equations hold:
\beqa
\frac{\lambda}{2} - \kk(f) &=& \text{Ric}(\kk,\kk)\ =\ 0,\label{eqn:lambda5a}\nonumber\\
\mm(f) - f \kappa &=& \text{Ric}(\kk,\mm)\ =\ 0,\phantom{\frac{1}{2}}\label{eqn:kappa5}\nonumber\\
\lambda - f (\text{div}\,\kk) &=& \text{Ric}(\mm,\mb)\ =\ S,\label{eqn:div5a}\phantom{\frac{1}{2}}\\
f \sigma &=& \text{Ric}(\mm,\mm)\ =\ 0.\phantom{\frac{1}{2}}\label{eqn:sigma5a}
\eeqa
Since $\sigma = 0$, \eqref{eqn:sigma5a} is satisfied; since $\text{div}\,\kk = 0$, so is \eqref{eqn:div5a}, provided that we set $\lambda = S$ (recall that the scalar curvature is a constant).  Since $\kappa = 0$, we must now show that $f$ can be chosen to satisfy $\kk(f) = S/2$ and $\mm(f) = 0$.

Since $\nabla \kk = 0$ and $(M,g)$ is complete and simply connected, the existence of $f$ is guaranteed by the de Rham decomposition theorem, which says that $(M,g)$ splits isometrically as $\RR \times S^2$, with $\kk$ defining the vertical direction.  For such an $f$, $(M,g,f\kk)$ is then a non-Einstein shrinking Ricci soliton with $\lambda = S$ (non-Einstein because $(\mathscr{L}_{f\kk}\,g)(\kk,\kk) = S/2 \neq 0$).
\end{proof}

\section*{Acknowledgements}
This work was supported by the World Premier International Research Center Initiative (WPI), MEXT, Japan; this manuscript first appeared when both authors were members of  the Kavli Institute for the Physics and Mathematics of the Universe (IPMU), University of Tokyo, Japan.
We kindly thank Benjamin Schmidt, Wolfgang Ziller, and Robert Bryant for very helpful discussions.

\bibliographystyle{siam}
\bibliography{A-MB_AdvGeom}
\end{document}